\documentclass{amsart}

\usepackage[T1]{fontenc}
\usepackage{lmodern}
\usepackage{graphicx}
\usepackage[english]{babel}
\usepackage{mathrsfs}
\usepackage[bbgreekl]{mathbbol}
\usepackage{amsmath, amsfonts, amssymb, dsfont, amsthm}
\usepackage{stmaryrd}
\usepackage{mathrsfs}
\usepackage{mathtools}
\usepackage{listings}
\usepackage{floatrow}
\usepackage{fourier-orns}
\usepackage{textcomp}
\usepackage{faktor}
\usepackage{xfrac}
\usepackage[all]{xy}
\usepackage[colorlinks = True , allcolors = blue]{hyperref}

\newcommand{\im}{\operatorname{im}}
\newcommand*\diff{\mathop{}\!\mathrm{d}}
\newcommand{\Diff}{\operatorname{Diff}}

\newcommand{\Id}{\operatorname{Id}}

\newcommand{\R}{\mathbb R}

\newcommand{\Z}{\mathbb Z}
\newcommand{\E}{\mathbb E}
\newcommand{\CC}{\mathbb C}

\newcommand{\CCC}{\mathscr{C}}

\newcommand{\F}{\mathcal{F}}

\newcommand{\feui}{\mathcal F}

\newcommand{\End}{\operatorname{End}}

\newcommand{\gr}{\operatorname{gr}}
\newcommand{\ad}{\operatorname{ad}}

\newcommand{\g}{\mathfrak{g}}
\newcommand{\gt}{\mathfrak{t}}
\newcommand{\p}{\mathfrak{p}}
\newcommand{\n}{\mathfrak{n}}

\newcommand{\dn}{\diff^{\nabla}}

\newcommand{\act}{\curvearrowright}

\newcommand{\ttt}{\mathfrak{t}}

\newcommand{\ho}{\operatorname{Hol}}

\newcommand\isomto{\stackrel{\sim}{\smash{\longrightarrow}\rule{0pt}{0.4ex}}}

\theoremstyle{plain}
\newtheorem{Theorem}{Theorem}[section]

\newtheorem{Lemma}[Theorem]{Lemma}
\newtheorem{Proposition}[Theorem]{Proposition}
\newtheorem{Corollary}{Corollary}[Theorem]
\newtheorem{Example}[Theorem]{Example}
\newtheorem{Examples}[Theorem]{Examples}

\theoremstyle{definition}
\newtheorem{Definition}[Theorem]{Definition}

\theoremstyle{remark}
\newtheorem{Remark}[Theorem]{Remark}

\DeclareSymbolFontAlphabet{\mathbb}{AMSb}
\DeclareSymbolFontAlphabet{\mathbbl}{bbold}

\title[Transverse parabolic structures \& BGG sequences]{Transverse parabolic structures and transverse BGG sequences}
\author{Clément Cren}
\address{Mathematisches Institut\\
Georg-August Universität Göttingen\\
Bunsenstraße 3-5\\
D-37073 Göttingen\\
Deutschland}
\email{\href{mailto:clement.cren@mathematik.uni-goettingen.de}{clement.cren@mathematik.uni-goettingen.de}}

\begin{document}

\begin{abstract}
Manifolds endowed with a parabolic geometry in the sense of Cartan come with natural sequences of differential operators and their analysis provide the so called (curved) BGG sequence of \v{C}ap, Slov\'{a}k and Sou\v{c}ek. The sequences involved do not form an elliptic complex in the sense of Atiyah but enjoy similar properties. The proper framework to study these operators is the filtered calculus associated to the natural filtration of the tangent bundle induced by the parabolic geometry. Such analysis was carried over by Dave and Haller in a very general setting. In this article we use their methods associated with the transversal index theory for filtered manifolds developed by the author in a previous paper to derive curved BGG sequences for foliated manifolds with transverse parabolic geometry.
\end{abstract}

\maketitle

\section{Introduction}

We pursue the ideas of the transversal Rockland condition in two directions: we replace operators by complexes and Rockland condition by the graded Rockland condition in the sense of \cite{dave2017graded}. The motivation is the creation of a Bernstein-Gelfand-Gelfand type complex for foliated manifolds with a transverse parabolic geometry. 

The BGG sequences have an algebraic origin in the work of Bernstein, Gelfand, and Gelfand \cite{BGG}. It served as a projective resolution for representations of highest weight of semi-simple Lie groups. Using Kostant's approach to the Borel-Weil-Bott theorem \cite{Kostant}, they were reinterpreted in a more geometric context as a resolution of some sheaves of constant functions on generalized flag varieties. Čap, Slovák, and Souček pushed further this approach, interpreting the generalized Verma modules used in the BGG resolution as jet bundles over the generalized flag varieties. This observation allowed them to generalize BGG sequences to manifolds with parabolic geometries in the sense of Cartan \cite{CurvedBGGArticle}, i.e. manifolds "locally modeled" over a specific generalized flag variety. 

Here we try to generalize these ideas to foliated manifolds with transverse parabolic geometry, i.e. the leaf space is "locally modeled" over the generalized flag varieties. Since the BGG sequence can be constructed from the twisted de Rham complex on a tractor bundle, we show that in the context of transverse parabolic geometries, the twisted transverse de Rham complex shares similar properties and is modeled (i.e. on the symbolic level) on the Chevalley-Eilenberg complex with values in the representation used to build the tractor bundle. 

More precisely, given a semi-simple Lie group $G$, a parabolic subgroup $P$ and a foliated manifold $M$ with transverse $(G,P)$ geometry (in a looser sense than in the literature), and a $G$-representation $\E$, we show that $M$ is a foliated filtered manifold in the sense of \cite{cren2022transverse}, and study a sequence of operators, the transversal twisted de Rham complex, $(\Omega^{0,\bullet}(M,E),\diff^{\nabla}_N)$. Here $E \to M$ is the tractor bundle associated to the representation $\E$ of $G$ and $\nabla$ is the corresponding tractor connection. The subscript \(N\) indicates that the de Rham differential is here restricted to the directions that are normal to the foliation. We then show that under a regularity condition on the geometry, the transversal osculating groupoid is identified to the trivial bundle with fiber the opposed nilpotent Lie group. We also show that under this condition the transversal complex becomes the Chevalley-Eilenberg cohomological complex at the transversal symbolic level. We can thus show that our sequence of operators satisfies the transversal (graded) Rockland condition. 

Using a Kostant type co-differential on the transverse forms, we define a sequence of BGG type operators on the corresponding homology bundles:
$$D_{\bullet} \colon \mathcal{H}_{\bullet} := (M \times H_{\bullet}(\p_+,\E)) \to \mathcal{H}_{\bullet +1}.$$
Using the result on the transverse de Rham sequence we prove that the transverse BGG sequence $(\mathcal{H}_{\bullet},D_{\bullet})$ is also transversally graded Rockland. Moreover, if the transverse BGG sequence is a complex, then so is the BGG sequence, and the natural map from de Rham to BGG becomes an isomorphism in cohomology.

This construction serves several purposes. First, BGG sequences have been proved useful in representation theory: from their invention to their use to construct Kasparov's $\gamma$-element, and several proofs of the Baum-Connes conjecture (see \cite{Chen,JulgKasparov,JulgSp} for the rank one case and \cite{YunckenBGG} for an example in higher rank). Secondly, it is well known that Cartan geometries can be used to model a large class of geometric structures. The parabolic ones represent projective and CR geometries among others. Our work can also be seen as an attempt to develop such geometries and index theoretic invariants on their spaces of leaves. Finally this work also serves as a large class of examples for the tools developed in our previous paper \cite{cren2022transverse}.

\section{Kostant's Hodge theory and the origins of the BGG machinery}\label{Algebraic}

We first recall the classical construction Kostant's Laplacian for a semi-simple Lie group $G$ and a parabolic subgroup $P$, as well as the associated Hodge theory. This construction was introduced in \cite{Kostant} as a way to prove the Borel-Weil-Bott theorem (a way to construct irreducible representations of a certain highest weight for a semi-simple Lie group) and is purely algebraic. It involves the Lie algebra homology and cohomology introduced in \cite{ChevalleyEilenberg,Koszul}, to perform a construction similar to the Hodge theory in Riemannian geometry. This analogy has a more concrete incarnation as the differential will be the local model (i.e. symbol) of a differential operator on the homogeneous space (and later on spaces with parabolic geometries). 

Let $G$ be a (real or complex) semi-simple Lie group and $P \subset G$ a parabolic subgroup. Denote by $\g$ and $\p$ their respective Lie algebras. We assume the Lie algebra $\g$ to be $|k|-$graded i.e.
$$\g = \bigoplus_{i = -k}^k \g_i, \ \forall i,j, [\g_i,\g_j] \subset \g_{i+j},$$
such that $\p = \bigoplus_{i\geq 0} \g_i$. Such a decomposition always exists and only needs an appropriate choice of Cartan subalgebra and positive root system (see \cite{CurvedBGGBook} for instance). 

The Killing form on $\g$ being non-degenerate, it induces isomorphisms of $\g_0$-modules $\g_i \cong \g_{-i}^*$ for $i \neq 0$. In particular we get an isomorphism of \(\g_0\)-modules $\left(\faktor{\g}{\p}\right)^* \cong \p_+$ where $\p_+ = \bigoplus_{i > 0}\g_i$ (not of Lie algebras since \(\p \subset \g\) is not an ideal). On the other hand, we can identify the quotient space $\faktor{\g}{\p}$ to the Lie algebra $\g_- = \bigoplus_{i < 0}\g_i$. These two identifications allow us to use the co-chain complex of $\g_-$ and the chain complex for $\p_+$ at the same time.
Indeed, let $\mathbb{E}$ be a $\g$-module so it can be restricted to a $\g_-$-module and a $\p_+$ module at the same time. Recall that the Chevalley-Eilenberg complex construction endows $\Lambda^{\bullet}\g_-^*\otimes \mathbb{E}$ with the structure of a co-chain complex with a de Rham type differential $\partial$ as co-boundary map. In a dual fashion, the complex $\Lambda^{\bullet} \p_+ \otimes \mathbb{E}$ has a structure of a chain complex with boundary map $\partial^*$. The duality result invoked earlier gives for each $n\geq 0$ an isomorphism $C_k(\p_+, \mathbb{E}) := \Lambda^k\p_+\otimes \mathbb{E} \cong \Lambda^k \g_-^* \otimes \mathbb{E} =: C^k(\g_-,\mathbb{E})$. 

Under these identification the maps $\partial$ and $\partial^*$ become adjoint to each other for natural inner products\footnote{If \(\g\) is real semi-simple, in the complex case we get a hermitian form.} and hermitian forms on $\g^*$ and $E$, see \cite{Kostant}\footnote{This justifies the notation $\partial^*$ for the homological differential.}. The map $\partial^*$ is called the Kostant co-differential. We can then form Kostant's Laplacian:
$$\Box_{\bullet} := \partial^*\partial + \partial\partial^* \colon C_{\bullet}(\p_+,\mathbb{E}) \to C_{\bullet}(\p_+,\mathbb{E}).$$

\begin{Lemma}Let $k\geq 0$ and $x \in \Lambda^k\left(\faktor{\g}{\p}\right)^*\otimes \mathbb{E}$. If $\partial \partial^*x = 0$ then $\partial^*x = 0$. If $\partial^*\partial x = 0$ then $\partial x = 0$.\end{Lemma}
\begin{proof}
Using the aforementioned inner product (or hermitian form) for which $\partial$ and $\partial^*$ are adjoint maps we immediately get the result. This inner product is constructed from the Killing form using a well chosen (in particular \(\g_0\)-equivariant) Cartan involution, see \cite{Kostant,CurvedBGGBook}.
\end{proof}

We can now perform a Hodge-type decomposition of the complex:

\begin{Theorem}[Kostant \cite{Kostant}]Let $n \geq 0$, we have the following decomposition, orthogonal for the Killing form:
$$\Lambda^n\left(\faktor{\g}{\p}\right)^*\otimes \mathbb{E} =  \im(\partial) \oplus \ker(\Box_n) \oplus \im(\partial^*).$$
Every (co)homology class in $H^n(\g_-,\mathbb{E})$ or $H_n(\p_+,\mathbb{E})$ has a unique harmonic (i.e. in $\ker(\Box_k)$) representative, this provide $\g_0$-equivariant sections of the natural projections of co-chains and chains $C^n(\g_-,\mathbb{E}) \to H^n(\g_-,\mathbb{E})$ and $C_n(\p_+,\mathbb{E})\to H_n(\p_+,\mathbb{E})$.
\end{Theorem}
\begin{proof}
We view elements of $\Lambda^k\left(\faktor{\g}{\p}\right)^*\otimes \mathbb{E}$ both as $k$-chains on $\p_+$ and $k$-co-chains on $\g_-$. To show the decomposition we first show that $\ker(\Box_k) = \ker(\partial)\cap\ker(\partial^*)$. If $x\in \ker(\Box_k)$ then we have $\partial\Box_k x = 0$ which gives $\partial\partial^*\partial x = 0$. Applying the previous lemma twice gives $\partial x =0$. The same reasoning with $\partial^*\Box_kx$ gives $\partial^*x = 0$ thus the inclusion $\ker(\Box_k)\subset \ker(\partial)\cap\ker(\partial^*)$. The converse inclusion is obvious by construction of $\Box_k$. Now another easy consequence of the lemma is that $\im(\partial)\cap\ker(\partial^*) = 0$ and $\im(\partial^*)\cap \ker(\partial) =0$. Combining these results we get $(\im(\partial)\oplus\im(\partial^*))\cap \ker(\Box_k) = 0$. We now have the three subspaces in direct sum and need to prove that they generate the whole space. From the inclusion $\im(\Box_k) \subset \im(\partial)\oplus\im(\partial^*)$ and the rank nullity theorem applied to $\Box_k$ we get the equality of dimensions between the sum of subspaces and the whole space and thus the Hodge decomposition.
It follows from the decomposition that $\ker(\partial^{(*)}) = \ker(\Box_k) \oplus \im(\partial^{(*)})$ and thus the isomorphisms (of $\g_0$-modules):
\[H_k(\p_+,\mathbb{E})\cong \ker(\Box_k) \cong H^k(\g_-,\mathbb{E}).\qedhere\]
\end{proof}

From this result we can explain the BGG machinery. The goal is to create a new complex that will compute the same cohomology. Denote by $\pi_k \colon \ker(\partial) \to H_k(\p_+,\mathbb{E})$ the natural projection and by \newline
$S_k \colon H_k(\p_+,\mathbb{E}) \isomto \ker(\Box_k) \subset C^k(\g_-,\mathbb{E})$ its section constructed in the previous theorem. The idea is now to define the Bernstein-Gelfand-Gelfand operator as:
$$D_k := \pi_{k+1}\partial_kS_k \colon H_k(\p_+,\mathbb{E}) \to H_{k+1}(\p_+,\mathbb{E}) $$
However on the algebraic level since $\ker(\Box_k) = \ker(\partial)\cap\ker(\partial^*)$ this just gives $D_k = 0$. Another way to see that this idea is too naive at the level of the Lie algebra is to use Kostant's decomposition of $H_k(\p_+,\mathbb{E})$ as a direct sum of irreducible representations of highest weight for $\g_0$. Since $H_k(\p_+,\mathbb{E})$ and $H_{k+1}(\p_+,\mathbb{E})$ involve different irreducible representations, there cannot be non-trivial homomorphisms of $\g_0$-modules between them. 

The trick at the algebraic level is to use the Verma modules associated to these representations instead. This was the original idea of Bernstein, Gelfand and Gelfand in \cite{BGG}. On the homogeneous space, Verma modules can be seen as homogeneous bundles using jets bundles. This observation allowed for a geometric formulation of the BGG machinery on manifolds with parabolic geometry in \cite{CurvedBGGArticle}. In this context the Kostant co-differential is still a bundle map but the differential $\partial$ becomes a differential operator (the de Rham differential with respect to some connection). Combining these two operators will give rise to a Hodge theory on differential forms and allow us to define the BGG operators as previously. This time however they will give a sequence of (non-trivial) differential operators. In the flat case (for the aforementioned connection) the de Rham sequence of operators is an actual co-chain complex and the BGG sequence is a complex as well. In this case, both compute the same cohomology.

\section{The geometric setting}\label{GeomSet}

We fix $G$ a semi-simple Lie group and $P \subset G$ a parabolic subgroup. We also assume the Lie algebra $\g$ to be $|k|-$graded i.e.
$$\g = \bigoplus_{i = -k}^k \g_i, \ \forall i,j, [\g_i,\g_j] \subset \g_{i+j}$$
such that $\p = \bigoplus_{i\geq 0} \g_i$. as in the previous section. If $M$ is a manifold endowed with a right $P$-action, we denote by $r$ the morphism $r \colon P \to \operatorname{Diffeo}(M)^{op}$ into the group of diffeomorphisms of the manifold. By differentiation, this induces a map $\p \to \mathfrak{X}(M)$. We denote this map by $\xi \mapsto X_{\xi}$, we have:
$$\forall \xi \in \p, \forall x \in M, X_{\xi}(x) = \frac{\diff}{\diff t}_{|t=0} r_{\exp(t\xi)}(x).$$
Note that this map corresponds to the the anchor of the action algebroid $M \rtimes \p \to M$.  

\begin{Definition}A transverse $(G,P)-$geometry is the data of $(M,\F,\omega)$ where $M$ is a manifold, $\F \subset TM$ is an integrable subbundle (i.e. a foliation) and $\omega \colon TM \to \g$ is a $\g$-valued 1-form with the following properties :
\begin{itemize}
\item[(i)] there is a proper smooth right action of $P$ on $M$ such that $\F$ is $P$-invariant and $M \rtimes \p \subset \F$ \footnote{We mean that the image of $M\rtimes \p$ by its anchor is included in $\F$. Condition $(iv)$ below implies that the $P$-action is locally free. We can thus identify $M\rtimes \p$ to its image by its anchor.}
\item[(ii)] $\ker(\omega) \subset \F \subset \omega^{-1}(\p)$
\item[(iii)] $\omega$ is $P$-equivariant i.e. $\forall h \in P, r_h^*\omega = Ad(h^{-1})\circ \omega$
\item[(iv)] $\forall \xi \in \p, \omega(X_{\xi}) = \xi$
\item[(v)] $\overline{\omega} \colon \faktor{TM}{\F} \to M \times \faktor{\g}{\p}$, which is defined by axiom (ii), is an isomorphism
\item[(vi)] $\forall X \in \Gamma(\ker(\omega)), L_X\omega = 0$
\end{itemize}
\end{Definition}

\begin{Example}If $\Gamma \subset G$ is a discrete subgroup then $\Gamma\backslash G$ has a transverse $(G,P)$-geometry for the Maurer-Cartan form. Here $\F$ is the foliation induced by the $P$-action. The leaf space then corresponds to $\Gamma\backslash G/P$.
\end{Example}

This definition is different from the usual one \cite{Blumenthal}. In the literature, a foliated manifold \((M_0,\feui_0)\) has a transverse \((G,P)\)-geometry if there exists a principal \(P\)-bundle \(M\to M_0\) whose total space is endowed with a foliation \(\overline{\feui_0}\) which is \(P\)-invariant, intersects the bundle of vertical vectors trivially, and projects onto \(\feui_0\). The total space of the bundle is also endowed with a \(\g\)-valued 1-form satisfying similar properties as in the definition before. With this setting, people focus on the manifold \(M_0\) and the principal bundle acts as some extra data.
Here we focus on the transverse structure. Because of this, we can work directly with the total space of the principal bundle, \(M\). Indeed, if we denote by \(\feui:= \overline{\feui_0}\oplus \{X_{\xi}, \ \xi \in \p\}\subset TM\), we obtain a new foliation. The holonomy groupoid (see Section \ref{Holonomy}) of \(\feui\) is Morita equivalent to the one of \(\feui_0\), which means they have the same transverse geometry. This generalization allows to consider non-necessarily free actions of \(P\) on \(M\). Because of axiom (iv) however, they are at least locally free. When the action is free we can quotient out the manifold and its foliation, \((M_0,\feui_0) := \left(\faktor{M}{P},\faktor{\ker(\omega)}{P}\right)\) (or take \(\faktor{\feui}{P}\) directly, the foliation is the same since the \(X_{\xi}, \xi \in \mathfrak{p}\) are tangent to the \(P\)-orbits in \(M\)). The foliated  manifold \((M_0,\feui_0)\) then has a transverse \((G,P)\)-geometry in the sense of \cite{Blumenthal}.

Regardless of the definition, the intuition behind these spaces is that the space of leaves \(\faktor{M}{\feui}\) has to be seen as a space that "locally looks like" the homogeneous space \(\faktor{G}{P}\). By this we mean that the bundle of space of leaves \(\faktor{M}{\ker(\omega)} \to \faktor{M}{\F}\) locally mimics the behavior of the principal \(P\)-bundle \(G\to \faktor{G}{P}\) with the Cartan connection replacing the Maurer-Cartan form.



\begin{Examples}We give the geometric structures on the space of leaves corresponding to the pair $(G,P)$:
\begin{itemize}
\item $G = PGL(n+1)$ and $P$ the isotropy group of a line correspond to projective geometry
\item $G = O(n+1,1)$ and $P$ the stabilizer of a point on the conformal sphere correspond to conformal geometry
\item $G = SU(p+1,q+1)$ and $P$ the stabilizer of an isotropy line in $\CC^{p+q+2}$ correspond to CR geometry
\item $G = SP(n+1)$ and $P$ the stabilizer of a line in $\R^{2(n+1)}$ correspond to projective contact geometry
\end{itemize}
\end{Examples}

\begin{Proposition}$\ker(\omega)$ is a foliation.\end{Proposition}  
\begin{proof}
Let $X,Y \in \Gamma(\ker(\omega))$, 
\begin{align*}
\omega([X,Y]) &= X\cdot\omega(Y) - Y\cdot\omega(X) - \diff\omega(X,Y)  \\
			  &= -\diff\omega(X,Y) \\
			  &= -\iota_X\diff\omega(Y) \\
			  &= -L_X\omega(Y) \\
			  &= 0. \tag*{\qedhere}
\end{align*}
\end{proof}

\begin{Corollary}The foliation $\F$ splits into the direct sum:
$$\F = \ker(\omega) \oplus M\rtimes \p.$$
Both are subalgebroids.
\end{Corollary}
\begin{proof}
By axiom $(iv)$ we know the sum is direct. We already know the inclusion $\supset$. The other is a consequence of axiom $(v)$.
\end{proof}

\begin{Definition}The curvature form of $\omega$ is $K = \diff \omega + [\omega,\omega]_{\g} \in \Omega^2(M,\g)$.\end{Definition}

Let us now define the filtration of $TM$ induced by the transverse parabolic geometry and the associated osculating nilpotent Lie algebra bundle. For $i \in \Z$ let $\g^i = \bigoplus_{j \geq i}\g_j$. We obtain a filtration $\cdots \supset \g^{i-1} \supset \g^i \supset \g^{i+1} \supset \cdots$, as well as the relations $\forall i,j\in \Z, [\g^i,\g^j] \subset \g^{i+j}$ and $\forall i, \g_i = \faktor{\g^i}{\g^{i+1}}$. Define a Lie bracket with:
\begin{align*}
[\cdot,\cdot] \colon \faktor{\g^{-i}}{\g^{-i+1}}\times \faktor{\g^{-j}}{\g^{-j+1}} &\to \faktor{\g^{-i-j}}{\g^{-i-j+1}} \\
						(X,Y)                                                  &\mapsto [X,Y]_{\g} \mod  \g^{-i-j+1}.
\end{align*}
It endows $\bigoplus_{i\leq 0}\g_i = \bigoplus_{i\leq 0} \faktor{\g^i}{\g^{i+1}}$ with a Lie algebra structure. This structure is not the one induced from $\g$ since $\p_+$ is not an ideal in $\g$. The subspace $\g_- = \bigoplus_{i< 0}\g_i$ is then an ideal and the quotient is isomorphic to \(\g_0\). As for \(\p_+\), since $\p$ is not an ideal in $\g$, the quotient map $\g \to \faktor{\g}{\p} \cong \g_-$ is not a Lie algebra homomorphism.

Let $TM = H^{-r} \supset \cdots \supset H^{-1} \supset H^0$ be a foliated filtration i.e. 
$$\forall i,j \left[\Gamma(H^{-i}), \Gamma(H^{-j})\right] \subset \Gamma(H^{-i-j}).$$
Using the Lie bracket of vector fields we get in the same fashion as before
$$[\cdot,\cdot] \colon \faktor{\Gamma(H^{-i})}{\Gamma(H^{-i+1})}\times \faktor{\Gamma(H^{-j})}{\Gamma(H^{-j+1})} \to \faktor{\Gamma(H^{-i-j})}{\Gamma(H^{-i-j+1})}$$
This bracket is tensorial in both variables for \(i,j\geq 1\). With the identifications 
$$\faktor{\Gamma(H^{-i})}{\Gamma(H^{-i+1})} = \Gamma\left(\faktor{H^{-i}}{H^{-i+1}}\right),$$ 
for $i \geq 1$, we get a fiberwise Lie bracket $[\cdot,\cdot] \colon \ttt_HM \to \ttt_HM$. Here we have denoted:
$$\gt_HM = H^{-1} \oplus \faktor{H^{-2}}{H^{-1}} \oplus \cdots \oplus \faktor{H^{-r}}{H^{-r+1}},$$
and considered \(H^0\) as a subspace of \(H^1\), ignoring things in degree 0 and considering them as elements of degree 1.
The algebroid $\ttt_HM$ is a bundle of nilpotent Lie algebras (\textit{a priori} not locally trivial). Denote by $T_HM$ the bundle of (connected, simply connected) nilpotent Lie groups integrating them using the Baker-Campbell-Hausdorff formula.

For all $x \in M$, $H^0_x \subset \ttt_{H,x}M$ is included in the center. Indeed, since \(H^0\) preserves the other bundles under Lie brackets of vector fields, they do not augment the degree in the filtration. Since they are considered as having order 1 however, the Lie bracket becomes 0 in the associated grading. We denote by $\ttt_{\sfrac{H}{H^0}}M$ the quotient Lie algebra bundle and $T_{\sfrac{H}{H^0}}M$ the corresponding Lie group bundle. Note that $\ttt_{\sfrac{H}{H^0}}M$ can be constructed directly using the previous associated grading construction with the vector bundle \(\faktor{TM}{H^0}\).

Back to our setting, the Cartan connection $\overline{\omega}$ yields an identification of vector bundles: 
$$\faktor{TM}{\F} \cong M\times \faktor{\g}{\p}.$$ 
The Lie brackets on vector fields and on $\g$ might however fail to be compatible. Requiring $\omega$ to be a Lie algebra homomorphism would be too strong. However, we can define a filtration on $TM$ and get weaker conditions in order to have $\ttt_{\sfrac{H}{H^0}}M \cong M \times \g_-$. Let $\overline{H}^{-i} = \overline{\omega}^{-1}\left(\faktor{\g^{-i}}{\p}\right) \subset \faktor{TM}{\F}$ and $H^{-i} \subset TM$ its lift, in particular $H^0 = \F$ and $H^{-r} = TM$.

\begin{Definition}A $(G,P)$-transverse geometry is called regular if the sub-bundles  $(H^{-i})_{i\geq 0}$ form a foliated Lie filtration and $\ttt_{\sfrac{H}{H^0}}M \cong M \times \g_-$\footnote{As $P$-equivariant vector bundles where the action $P\act \g_-$ is trivial.} as bundles of Lie algebras.\end{Definition}

Although the requirement of trivial bundle might seem strong, remember that we are working on a generalization of the total space of a principal \(P\)-bundle. On the model case \(G \to \faktor{G}{P}\), the bundles constructed on the homogeneous space are quotients by the \(P\)-action of trivial bundles on \(G\). The same happens for manifolds with \((G,P)\)-geometry, see \cite{CurvedBGGBook}. Because of this, all the bundles considered on \(M\) will be trivial. However, they will need to be equivariant for the \(P\)-action, and also for the holonomy action (see Section \ref{Holonomy action} below). This way, they have to be understood as non-trivial bundles over the leaf space \(\faktor{M}{\feui}\).

We now investigate geometric conditions to ensure the regularity of a transverse \((G,P)\)-geometry.

\begin{Lemma}Let \((H^{-i})_{i\geq 0}\) be the filtration induced by a transverse \((G,P)\)-geometry on \(M\), then: $$\forall i \geq 0, \left[\Gamma(H^{-i}),\Gamma(H^0)\right] \subset \Gamma(H^{-i}).$$\end{Lemma}
\begin{proof}
Let $i\geq 0$ and $X \in \Gamma(H^0), Y \in \Gamma(H^{-i})$ we need to show that $\omega([X,Y]) \in \CCC^{\infty}(M,\g^{-i})$. We have
$$\omega([X,Y]) = X\cdot \omega(Y) - Y\cdot \omega(X) - \diff \omega(X,Y),$$
but $\omega(Y) \in \CCC^{\infty}(M,\g^{-i})$ and $\g^{-i}$ being a vector subspace of $\g$, then 
$$X\cdot\omega(Y) \in \CCC^{\infty}(M,\g^{-i}).$$
Likewise, $\omega(X) \in \CCC^{\infty}(M,\p) \subset \CCC^{\infty}(M,\g^{-i})$ as $-i \leq 0$ so we also have 
$$Y\cdot\omega(X) \in \CCC^{\infty}(M,\g^{-i}).$$
Finally $\diff\omega(X,Y) = K(X,Y) - [\omega(X),\omega(Y)]_{\g}$. We already know that the action of $\p$ preserves the filtration: $[\g^i,\p]\subset \g^i$. We thus need to show that $K(X,Y) \in \CCC^{\infty}(M,\g^{-i})$. We show a stronger result that will be useful later on.
\end{proof}

\begin{Lemma}\label{CurvBas}Let \((M,\feui)\) be a foliated manifold with transverse \((G,P)\)-geometry, \(K\) its curvature tensor. Then: 
$$\forall X \in \Gamma(\F), \iota_X K =0.$$ \end{Lemma}

\begin{proof}
We use the fact that $\F = M\rtimes \p \oplus \ker(\omega)$.

First case, $X \in \ker(\omega)$ : 
We have $\omega(X) = 0$ so the bracket part vanishes and $\iota_X\diff\omega = L_X\omega = 0$ by axiom $(vi)$, hence $\iota_XK = 0$.

Second case, $X = X_{\xi}, \xi \in \p$ :
$\iota_{X_{\xi}}\omega$ is a constant function so $\diff \iota_{X_{\xi}}\omega = 0 $ and $\iota_{X_{\xi}}\diff\omega = L_{X_{\xi}}\omega$.
\begin{align*}
L_{X_{\xi}}\omega &= \frac{\diff}{\diff t}_{|t=0} r^*_{\exp(t\xi)}\omega \\
				  &= \frac{\diff}{\diff t}_{|t=0} Ad(\exp(-t\xi))\circ \omega \\
				  &= - ad(\xi)\circ \omega \\
				  &= - [\xi,\omega] \\
				  &= - [\omega(X_{\xi}),\omega] \\
				  &= - \iota_{X_{\xi}}[\omega,\omega]_{\g}.
\end{align*}
Therefore $\iota_{X_{\xi}}K =  - \iota_{X_{\xi}}[\omega,\omega]_{\g}  + \iota_{X_{\xi}}[\omega,\omega]_{\g} =0$.
\end{proof}

\begin{Proposition}\label{CurvatureRegular}Let \((H^{-i})_{i\geq 0}\) be the filtration induced by a transverse \((G,P)\)-geometry on \(M\), then:
\begin{itemize}
\item[i)] $(H^{-i})_{i\geq 1}$ is a Lie filtration if and only if:
$$\forall i,j \geq 1, K(H^{-i},H^{-j}) \subset \g^{-i-j}.$$
\item[ii)] The $(G,P)$-transverse geometry is regular if and only if:
$$\forall i,j \geq 1, K(H^{-i},H^{-j}) \subset \g^{-i-j+1}.$$
\end{itemize}
\end{Proposition}

\begin{proof}
Let $i,j \geq 1$, $X \in \Gamma(H^{-i}), Y \in \Gamma(H^{-j})$:
\begin{align*}
\omega([X,Y]) &= X \cdot \omega(Y) - Y\cdot\omega(X) - \diff\omega(X,Y) \\
              &= X \cdot \omega(Y) - Y\cdot\omega(X) - [\omega(X),\omega(Y)] +K(X,Y).
\end{align*}
We have that $X \cdot \omega(Y) \in \CCC^{\infty}(M,\g^{-j}) \subset \CCC^{\infty}(M,\g^{-i-j})$. We also have that $Y\cdot\omega(X) \in \CCC^{\infty}(M,\g^{-i-j})$. We also have $[\omega(X),\omega(Y)] \in \CCC^{\infty}(M,\g^{-i-j})$ and thus deduce the first statement.

For the second statement we keep the same notations. The only thing to show is that $\omega([X,Y]) = [\omega(X),\omega(Y)] \mod \g^{-i-j+1}$. Since $i,j \geq 1$ we have $\CCC^{\infty}(M,\g^{-j}) \subset \CCC^{\infty}(M,\g^{-i-j+1})$ so $X\cdot\omega(Y) \in \CCC^{\infty}(M,\g^{-i-j+1})$. The same goes for $Y\cdot \omega(X)$ and we have 
$$\omega([X,Y]) = [\omega(X),\omega(Y)] + K(X,Y) \mod \g^{-i-j+1},$$ 
hence the result.
\end{proof}

\section{The transverse complex}\label{TransverseComplex}

Let $(\mathbb{E},\rho)$ be a finite dimensional $\g$-representation, denote by $E \to M$ the associated trivial bundle and by $\nabla = \diff + \rho\circ\omega \in \Omega^1(M,E)$ the canonical tractor connection. We want to mimic the curved BGG sequences of \cite{CurvedBGGArticle,CurvedBGGBook}. In order to do so we need the identification of the normal bundle with $\faktor{\g}{\p}$ and thus replace the usual complex of differential forms by the one of transverse forms. Let $N \subset TM$ be a subbundle such that $\F \oplus N = TM$, such a choice has no incidence on what follows. The space of differential forms then splits up into a bicomplex 
$$\Lambda^{i,j}T^*M = \Lambda^i N^* \otimes \Lambda^j \F^* \cong \Lambda^i \left(\faktor{TM}{\F}\right)^* \otimes \Lambda^j \F^*.$$ 
The de Rham differential splits up accordingly: 
$$\diff^{\nabla} = \sum_{i+j = 1} \diff^{\nabla}_{i,j},$$ 
with $\diff^{\nabla}_{i,j}(\Omega^{k,l}(M,E))\subset \Omega^{k+i,l+j}(M,E)$. Since $\diff^{\nabla}$ is determined by its image on zero-forms and one-forms then only the following terms remain:
$$\diff^{\nabla} = \diff^{\nabla}_{2,-1} + \diff^{\nabla}_{1,0} + \diff^{\nabla}_{0,1}+ \diff^{\nabla}_{-1,2}.$$
We denote by $\diff^{\nabla}_N = \diff^{\nabla}_{1,0}$ the differential in the transverse directions, $\diff^{\nabla}_{\F} = \diff^{\nabla}_{0,1}$ the differential in the longitudinal ones. For a vector field $X \in \mathfrak{X}(M)$, write $X = X^N + X^{F}$ for the decomposition corresponding to $\mathfrak{X}(M) = \Gamma(N) \oplus \Gamma(\F)$.

\begin{Proposition}$\diff^{\nabla}_{-1,2} = 0$ and $\diff^{\nabla}_{2,-1}$ is the contraction by $(2,-1)$-form $\Theta \in \Gamma(M,\Lambda^2 N^* \otimes \F)$ given by $\Theta(X^N,Y^N) = -[X^N,Y^N]^F$. In particular $\diff^{\nabla} = \diff^{\nabla}_N + \diff^{\nabla}_{\F}$ if and only if $N$ is involutive.
\end{Proposition}

\begin{proof}
As mentioned before, we only need to study the case of zero and one-forms. Let $X,Y \in \mathfrak{X}(M)$ be vector fields on $M$. Let $s \in \Gamma(E)$, 
$$\dn s(X) =\dn s (X^N)+ \dn s (X^F)   = \dn_N s(X^N) + \dn_{\F}s(X^F).$$
Let $\eta \in \Omega^{1,0}(M,E)$, we view $\eta$ as a 1-form vanishing on $\F$. Then
\begin{align*}
\dn\eta(X,Y) &= \dn\eta(X^N,Y^N) + \dn\eta(X^N,Y^F) \\
			   & \ \  + \dn\eta(X^F,Y^N)+ \dn\eta(X^F,Y^F) \\
               &= \dn_N\eta(X^N,Y^N) + \dn_{\F}\eta(X^N,Y^F) \\
               & \ \ + \dn_{\F}\eta(X^F,Y^N) + \dn_{-1,2}\eta(X^F,Y^F)
\end{align*}
hence $\dn_{-1,2}\eta(X^F,Y^F) = \nabla_{X^F}\eta(Y^F) - \nabla_{Y^F}\eta(X^F) -\eta([X^F,Y^F]) = 0$ as $\F$ is involutive.\\
Since $\dn_{-1,2}$ vanishes automatically on $\Omega^{0,1}(M,E)$ we get the first part of the statement.\\
For the second part, notice that we necessarily have $\dn_{2,-1}\eta = 0$ on $\Omega^{1,0}(M,E)$. The same happens with the contraction by $\Theta$. Let $\eta \in \Omega^{0,1}(M,\E)$ i.e. a 1-form vanishing on $N$. We have 
\begin{align*}
\dn_{2,-1}\eta(X^N,Y^N) &= \dn\eta(X^N,Y^N) \\
                          &= \nabla_{X^N}\eta(Y^N) - \nabla_{Y^N} \eta(X^N) - \eta([X^N,Y^N]) \\
                          &= -\eta([X^N,Y^N]^N + [X^N,Y^N]^F) \\
                          &= \eta(\Theta(X^N,Y^N)).
\end{align*}
Hence the result.
\end{proof}

The object of study in the next sections will be the sequence of differential operators:
$$\xymatrix{\Gamma(E) \ar[r]^-{\dn_N} & \Omega^{1,0}(M,E) \ar[r]^-{\dn_N} & \cdots \ar[r]^-{\dn_N} & \Omega^{n,0}(M,E)}$$
As for the de Rham sequence of operators, this is not a chain complex, $(\dn_N)^2 \neq 0$. The definition of $\dn_N$ involves choosing a complementary bundle to $\F$ but the operator pulled back between the bundles $\Lambda^i \left(\faktor{TM}{H^0}\right)^* \otimes E$ does not depend on these choices. Denote by $R^{\nabla}$ the curvature form associated to $\nabla$, it is the image of $K$ by $\g \to \End(\E)$.

\begin{Corollary}\label{Curvature Indentities}We have:
\begin{align*}
(\dn_{\F})^2 &= 0, \\
(\dn_N)^2 &= R^{\nabla} - (\dn_F\iota_{\Theta} + \iota_{\Theta}\dn_F), \\
\iota_{\Theta}^2 &= 0
\end{align*}\end{Corollary}
\begin{proof}
The proof is similar to the previous one. We square on both sides 
$$\dn = \dn_N+\dn_{\F} + \iota_{\Theta},$$
and identify the $(i,j)$-parts for $-4\leq i,j \leq 4, i+j=2$. The LHS is equal to $R^{\nabla}$. By Lemma \ref{CurvBas}, $R^{\nabla}$ is a form of bidegree $(2,0)$. Hence we get the following equalities:
\begin{itemize}
\item[(4,-2) :]$\iota_{\Theta}^2 = 0$
\item[(3,-1) :]$\dn_N\iota_{\Theta} + \iota_{\Theta}\dn_N = 0$
\item[(2,0) :]$(\dn_N)^2 + \dn_{\F}\iota_{\Theta} + \iota_{\Theta}\dn_{\F} = R^{\nabla}$
\item[(1,1) :]$\dn_{\F}\dn_N + \dn_N\dn_{\F} = 0$
\item[(0,2) :]$(\dn_{\F})^2 = 0$
\item[(-1,3) :]$0 = 0$
\item[(-2,4) :]$0 = 0$\qedhere
\end{itemize}
\end{proof} 

\begin{Remark}This sequence of operators can also be constructed for arbitrary foliated filtered manifold. The curvature is not concentrated in degree (2,0) in general however. Corollary \ref{Curvature Indentities} still holds with the (2,0)-component of the curvature replacing the whole curvature (and some other terms do not vanish anymore but they are not relevant for our study).\end{Remark}

\section{The holonomy groupoid of a foliation}\label{Holonomy}

The de Rham complex of a manifold with transverse parabolic geometry has no chance of being elliptic, even hypoelliptic, as it fails to capture what happens in the longitudinal directions. We actually have better chances to understand its index properties transversally to $\F$. In order to do so we need actions of the holonomy groupoid and the equivariance of the operators involved. The problem is that the holonomy groupoid only acts on the normal bundle to the foliation. To avoid this issue we will make sections of the algebroid act on the different bundles. For the $P$ part of the holonomy groupoid it can be reduced to a classical $P$-equivariance.

A Lie groupoid is given by the data of two manifolds \(G,M\), source and range maps \(r,s\colon G \to M\), a partially defined multiplication 
\[m\colon G^{(2)}= \{(g,h)\in G^{(2)}, s(g) = r(h)\} \to G,\] 
an involution (inversion map) \(\iota \colon G \to G\) and a unit map \(u \colon M \to G\). These maps satisfy the axioms of a category with objects \(M\) and arrows given by elements of \(G\), the element \(u(x), x \in M\) is the identity map of the object \(x\), \(s\) indicates the source object of a map, \(r\) its target  (so for instance \(r\circ u = \Id_M = s\circ u\))... We also ask these maps to be smooth, \(r,s,m\) to be submersions, \(\iota\) a diffeomorphism and \(u\) an embedding.

\begin{Examples}Here are some examples of Lie groupoids:
\begin{itemize}
\item Given a manifold \(M\) we can consider \(G = M\) and all the structure maps to be the identity.
\item A Lie group is a Lie groupoid over a point.
\item A smooth family of groups over a manifold \(\pi \colon G \to M\) is a groupoid with \(r = \pi = s\).
\item Given a manifold \(M\) we can consider the pair groupoid \(M \times M\) over \(M\). The source and range maps are the respective projections onto the second and first components. The unit map is the diagonal inclusion, the inversion swaps the two elements of a couple. Finally the composition law reads \(\forall x,y,z \in M, (x,y)\cdot (y,z) = (x,z)\).
\item The Poincaré groupoid of a manifold is the groupoid whose arrows from \(x\in M\) to \(y\in M\) is the set of homotopy classes of paths from \(x\) to \(y\) (with fixed ends). Composition is given by the concatenation of paths, the unit map by the trivial loops and the inversion is taking a path backwards.
\end{itemize}
\end{Examples}

Another example is the holonomy groupoid of a foliation. Let \((M,\feui)\) be a foliated manifold. Similarly to the Poincaré groupoid we can consider a groupoid over \(M\) whose arrows between \(x,y\in M\) are homotopy classes of paths tangent to the leaves between \(x\) and \(y\). In particular there is no arrow between points that are not on the same leaf. We can refine this groupoid to get something that only captures the transverse geometry of the foliation.
If a path from \(x\in M\) to \(y\in M\) is contained in a foliation chart, we can take local transversals around its ends. The trivialisation of the foliation in the chart and the path induce a diffeomorphism between the two transversals. Now given any path we can cover it with finitely many foliation charts. We can take points \(x=x_0, x_1, \cdots, x_m = y\) along the path, each one at the intersection of two following charts. Fixing small enough local transversals around these point, the path induces a diffeomorphism between two subsequent one of these. We can compose them to get a diffeomorphism between a local transversal around \(x\) and another one around \(y\). 
This diffeomorphism depends on the choices of charts and transversals but its germ doesn't. The same goes when we replace the path by a homotopically equivalent one. The germ of diffeomorphism along local transversals is called the holonomy of a path. We can quotient out paths having the same holonomy to obtain a new Lie groupoid over \(M\), the holonomy groupoid \(\ho(\feui)\).

This groupoid integrates the foliation in the sense that its Lie algebroid (the infinitesimal counterpart of Lie groupoids, in the same way that Lie algebras are the infinitesimal counterpart of Lie groups) is the foliation \(\feui\) with the Lie bracket of vector fields on its sections. The holonomy groupoid is also minimal in that sense, see \cite{debord2000groupoides}.

Since the algebroid is given by the foliation itself, it is easy to produce (local) bisections of this groupoid, i.e. maps \(\varphi \colon M \to \ho(\feui)\) with \(s\circ \varphi = \Id_M\) and \(r\circ \varphi\) is a local diffeomorphism. To this end, consider a local section \(X \in \Gamma_{loc}(M,\feui)\). Its flow at time \(1\) defines a local diffeomorphism of \(M\). Since \(\feui\) is a foliation, this flow preserves the leaves and define holonomies by the argument used above for paths. Therefore \(\exp(X)\) is a local bisection of \(\ho(\feui)\).

Finally the holonomy groupoids acts on the normal bundle of the foliation. Indeed if \(\gamma\in \ho(\feui)\) is an arrow between \(x\) and \(y\) then it represents a germ of diffeomorphims between local transversals around those points. Taking the differential of this germ at \(x\) gives a linear map \(\faktor{T_xM}{\feui_x}\to \faktor{T_yM}{\feui_y}\). These maps are compatible with the product and inversion law of \(\ho(\feui)\), this is a groupoid action \(\ho(\feui)\act \faktor{TM}{\feui}\) which preserves the vector space structure on the fibers.

\section{Invariant differential operators for the holonomy action}\label{Holonomy action}

\begin{Definition}Let $(M,\F)$ be a foliated manifold and $D \colon \Gamma(M,E) \to \Gamma(M,F)$ a differential operator between $\ho(\F)$-equivariant vector bundles $E,F$ over $M$. We say that $D$ is $\ho(\F)$-equivariant if for every (local) section $X \in \Gamma(M,\F)$ then:
$$\exp(X) \circ D = D \circ \exp(X).$$
\end{Definition}

\begin{Proposition}If $D \colon \Gamma(M,E) \to \Gamma(M,F)$ is $\ho(\F)$-equivariant then its principal transverse symbol (whether in the classical or filtered calculus if there is a filtration, see Section \ref{FilteredCalculus} below) is $\ho(\F)$-equivariant as a function on \(\feui^{\perp}\setminus\{0\}\) (or as a distribution on \(T_{\sfrac{H}{\feui}}M\) in the filtered case).
\end{Proposition}
\begin{proof}
Let \(x,y\in M, \gamma \in \ho(\feui)\) an arrow from \(x\) to \(y\). We can choose a local vector field \(X\in \Gamma_{loc}(M,\feui)\) such that \(\exp(X)(x) = y\) and \(\exp(X)\) represents the germ \(\gamma\). Now using the equivariance of \(D\) and taking the principal symbol on both sides, we obtain \(\gamma\circ \sigma_x(D) = \sigma_y(D)\circ\gamma\). In the last equation the element \(\gamma\) is seen as acting on the normal bundle of the foliation (or its filtered analog in the filtered case).
\end{proof}

Recall from Section \ref{GeomSet} that $M$ has a transversally filtered structure induced by its transverse Cartan geometry. We will now assume that the transverse Cartan geometry is regular. It is proven in \cite{cren2022transverse} that the action $\ho(H^0) \act \faktor{TM}{H^0}$ can be refined to an action $\ho(H^0)\act T_{\sfrac{H}{H^0}}M$ preserving the groupoid structure of $T_{\sfrac{H}{H^0}}M$. Here we have decomposed the foliation into two parts $H^0 = \F = M \rtimes \p \oplus \ker(\omega)$. To understand the different actions of $\ho(H^0)$ in our case, we will thus need to understand the actions of $P$ and $\ho(\ker(\omega))$.\\
Let us fix a \(G\)-representation \((\mathbb E,\rho)\) and consider the corresponding tractor bundle \(E = M \times \mathbb{E} \to M\) as in Section \ref{TransverseComplex}. The holonomy groupoid acts on $E$ in the following manner: $P$ acts on $M$ on the right via $r$ and on $\E$ on the left via $\rho$. This endows $E$ with the diagonal action: $(x,e)\cdot h = (r_h(x),\rho(h^{-1})(e))$. Similarly, $\ho(\ker(\omega))$ acts naturally on $M$ and trivially on $\E$ and the action on $E$ is the diagonal one. Let $P$ act on $\g$ on the left by the restriction of the adjoint action. Then, $P$ acts on the trivial bundle $M \times \g$ by $(x,\xi)\cdot h = (r_h(x) , \ad(h^{-1})(\xi))$. It makes $\omega \colon TM \to M\times \g$ a $P$-equivariant morphism by axiom $(iii)$. We make $\ho(\ker(\omega))$ act trivially on $\g$ and thus obtain an action of $\ho(H^0)$ on $M \times \g$ (which extends to the associated bundles $M\times G, M \times \g^*,\cdots$).

\begin{Remark}The action of $\ho(\ker(\omega))$ on $E$ or $\End(E)$ is given by the pullback of sections: $s \cdot \exp(X) = s\circ \exp(X) = \exp(X)^*s$. In this last equation, the value at \(x\in M\) only depends on the holonomy germ induced by \(\exp(X)\) at \(x\). Therefore we have a well defined groupoid action \(\ho(\ker(\omega))\act E\). Combining with the natural \(P\) action we obtain \(\ho(H^0)\act E\).\end{Remark}

\begin{Lemma}$\overline{\omega}$ is  $\ho(H^0)$-equivariant, hence $\rho \circ \omega$ is $\ho(H^0)$ equivariant. \end{Lemma}
\begin{proof}
The equivariance for the $P$ action is clear from the previous paragraph. We now need to prove the equivariance for elements of the form $\exp(X)$ with $X\in \Gamma(\ker(\omega))$.
Let $X$ be a section of $\ker(\omega)$, $t\mapsto \gamma_t = \exp(tX)$ the associated flow. Since $\gamma_{t*}X \in \Gamma(\ker(\omega))$ then: \(\frac{\diff}{\diff t}\gamma_t^*\omega = \gamma_t^*\left(L_{\gamma_{t*}X}\omega\right) = 0\). Therefore $\exp(X)^*\bar{\omega} = \exp(0)^*\bar{\omega} = \bar{\omega}$. The equivariance then follows from the action of $\ho(\ker(\omega))$ on the fibers of $E$ being trivial.
\end{proof}

\begin{Proposition}$\nabla$ is a $\ho(H^0)$-equivariant connection\end{Proposition}
\begin{proof}
We consider $\nabla \colon \Gamma(E) \to \Gamma(T^*M\otimes E)$. The action on the latter space of sections is the following: for elements $h$ of $P$ we have $\eta \cdot h = \rho(h^{-1})r_h^*\eta$, for elements $\exp(X)$ with $X \in \Gamma(\ker(\omega))$ we have $\eta\cdot \exp(X) = \exp(X)^*\eta$.
\paragraph*{First case: the $P$-action.}

Let $h \in P$, we have $s\cdot h = \rho(h^{-1}) r_h^*s$. Since the $P$-action is constant on the fibers of $E$ we have:
\begin{align*}
\diff(s\cdot h) &= \diff (\rho(h^{-1}) r_h^*s ) \\
          &= \rho(h^{-1}) (\diff r_h^* s) \\
          &= \rho(h^{-1}) r_h^*\diff s \\
          &= (\diff s)\cdot h.
\end{align*}
The equivariance of $\nabla$ with respect to the $P$-action is thus a consequence of the previous lemma.
\paragraph*{Second case: the $\ho(\ker(\omega))$-action.}

The action of a section of $\ker(\omega)$ is given by the pullback of sections and thus commutes with $\diff$. The equivariance of $\nabla$ is again a consequence of the lemma.
\end{proof}

\begin{Corollary}\label{Equivariance}The differential operators
$$\dn_N \colon \Omega^{i,0}(M,E) \to \Omega^{i+1,0}(M,E), i\geq 0,$$ 
are $\ho(H^0)$-equivariant.\end{Corollary}

\section{Pseudodifferential calculus on filtered manifolds}\label{FilteredCalculus}

In this section we review the ideas behind calculus on filtered manifold. We don't really need the whole pseudodifferential calculus for what follows as we mainly use differential operators. It is good to have in mind that there is a bigger class of operators than the differential operators. Indeed the Rockland condition (or ellipticity for the usual (pseudo)differential calculus) serves as an invertibility criterion for the principal symbol. The order of a product of symbols is the sum of the orders. The inverse of the symbol of a differential operator should thus have non-positive order. It therefore cannot be the symbol of a differential operator who has non-negative order (unless both have order \(0\), which means they are functions). This section won't contain any proof. We refer the reader to the papers \cite{vanErpYuncken,dave2017graded,cren2022transverse} where these constructions are carried out with more details.

Let \(((M,(H^{-1},\cdots,H^{-r})))\) be a filtered manifold (we don't need any foliation for the moment). We consider vector fields that are sections of \(H^{-k}\) as (filtered) differential operators of order \(k\). This extends to a new definition of order for any differential operator. The condition on Lie brackets of vector fields for filtered manifolds ensures that the order of a product is (at most) the sum of the orders.

Taking a differential operator \(D\) of order \(k\) and a point \(x\in M\), we can freeze the coefficients of \(D\) at the point \(x\) and only keep the part of order \(k\) (for the new filtered order). This gives a equivariant differential operator on the group \(T_{H,x}M\) (remember that the osculating groupoid \(T_HM\) constructed in Section \ref{GeomSet} is a family of groups). We can equivalently see this element as part of the universal enveloping algebra of \(\mathcal{U}(\ttt_{H,x}M)\). The collection \((D_x)_{x\in M}\) of such operators is called its principal symbol in the filtered calculus and gives a section of the universal enveloping algebroid \(\mathcal{U}(\ttt_HM) = \bigsqcup_{x\in M}\mathcal{U}(\ttt_{H,x}M)\).

Let us denote by \(\Diff_H(M)\) the algebra of differential operators filtered by the order given by the filtration. The principal symbol map \(\sigma \colon \Diff_H(M) \to \Gamma(M,\mathcal{U}(\ttt_HM))\) is a morphism of filtered algebras\footnote{We will sometimes write \(\sigma^k\) for operators of order \(k\). The map \(\sigma^k\) vanishes on symbols of lower order.}. The filtration on the latter space comes from a grading on \(\mathcal{U}(\ttt_HM)\). This map is isomorphic on the associated gradings. This comes from the following exact sequences for \(k\in \mathbb{N}\):
\[\xymatrix{0 \ar[r] & \Diff^{k-1}_H(M) \ar[r] & \Diff^k_H(M) \ar[r]^-{\sigma^k} & \Gamma(M,\mathcal{U}(\ttt_HM)_{-k}) \ar[r] & 0.}\]
Here \(\mathcal{U}(\ttt_HM)_{-k}\) denotes the subspace of elements of degree \(k\). It is spanned, at a given \(x\in M\), by products of the form \(X_1\cdots X_j\) with \(X_{\alpha} \in \faktor{H_x^{-\ell_{\alpha}}}{H_x^{-\ell_{\alpha}+1}} \subset \ttt_{H,x}M\) with \(\sum_{\alpha = 1}^j \ell_{\alpha} = k\).

\begin{Example}
    When the filtration is trivial (\(H^{-1} = TM\)) we recover the usual notion of principal symbol by taking a fiberwise Fourier transform.
\end{Example}

More generally, we can allow for larger classes of symbols by taking all families of distributions on the groups \(T_{H,x}M, x \in M\) that are quasi-invariant for the inhomogeneous dilations \((\delta_{\lambda})_{\lambda>0}\) on the fibers (see e.g. \cite{vanErpYuncken}). These dilations can be defined on the Lie algebras first. For \(\lambda > 0\), let
\[\delta_{\lambda} \colon \ttt_HM \to \ttt_HM, \ \delta_{\lambda \left|\sfrac{H^{-k}}{H^{-k+1}}\right. } = \lambda^k \Id.\]
These dilations preserve the Lie algebra structure on each fiber and can thus be lifted to Lie groups automorphisms on each fiber (i.e. groupoid automorphisms for \(T_HM\)).
The principal symbol of a differential operator of order \(k\) in the filtered calculus is homogeneous of order \(k\) for these dilations. The more general symbols are not homogeneous on the nose. To consider the principal part of such symbols however one needs to consider its equivalence class modulo \(\CCC^{\infty}_c(T_HM)\). These equivalence classes are homogeneous for the inhomogeneous dilations acting on the quotient. We denote by \(\Sigma^k(T_HM)\) the space of (equivalence classes of) such distributions. These distribution arise as principal symbols of new operators: the pseudodifferential operators in the filtered calculus. They can be obtained by "exponentiating" a symbol to an operator on \(\CCC^{\infty}(M)\), or characterized using a variant of the tangent groupoid, following \cite{vanErpYuncken}. We denote by \(\Psi^m_H(M)\) the space of pseudodifferential operators of order \(m\) (not necessarily an integer anymore) in the filtered calculus.

We now describe the Rockland condition. Given an operator \(P \in \Psi_H^m(M)\), we consider its principal symbol \(\sigma^m(P) \in \Sigma^m(T_HM)\). For a point \(x\in M\) and a (non-trivial) irreducible, unitary representation \(\pi \in \widehat{T_{H,x}M}\setminus\{1\}\), we can form an operator \(\diff\pi(\sigma^m(P)) \colon \mathcal{H}^{\infty}_{\pi} \to \mathcal{H}^{\infty}_{\pi}\). Here \(\mathcal{H}^{\infty}_{\pi} \subset \mathcal{H}_{\pi}\) is the subset of smooth vectors in the Hilbert space underlying the representation. The operator \(P\) satisfies the Rockland condition if all these operators (for any given \(x,\pi\) as above) are left invertible. This is equivalent to the existence of a symbol \(\eta \in \Sigma^{-m}(T_HM)\) such that \(\eta \sigma^m(P) = 1\). From this symbol \(\eta\), one can then construct a (left-)parametrix for \(P\) in the filtered calculus (a left inverse modulo regularising operators), see \cite{dave2017graded}.

We finally summarize the ideas of our previous work \cite{cren2022transverse}. Given a foliated filtered manifold \((M,H^0\subset \cdots \subset H^{-r} =TM)\), we can consider the osculating groupoid \(T_HM\) but also its transverse analog \(T_{\sfrac{H}{H^0}}M\) as described in Section \ref{GeomSet}. The quotient map \(T_HM \to T_{\sfrac{H}{H^0}}M\) induces, by push-forward of distributions, maps:
\[\int_{H^0} \colon \Sigma^{\bullet}(T_HM) \to \Sigma^{\bullet}(T_{\sfrac{H}{H^0}}M).\]
Here the spaces \(\Sigma^{\bullet}(T_{\sfrac{H}{H^0}}M)\) are constructed in a similar fashion to \(\Sigma^{\bullet}(T_HM)\). Indeed this construction only used the fact that \(T_HM\) was a smooth family of graded nilpotent groups. Therefore we just need to consider quasi-invariant distributions on the corresponding groupoid \(T_{\sfrac{H}{H^0}}M\).

Without filtrations, and after taking a Fourier transform on both sides, this map corresponds to the restriction map \(\CCC^{\infty}(T^*M\setminus 0) \to \CCC^{\infty}((H^0)^{\perp} \setminus 0)\).
Using the non-trivial irreducible unitary representations of \(T_{\sfrac{H}{H^0}}M\), we can define a transversal Rockland condition. We showed in \cite{cren2022transverse} that operators satisfying this transversally Rockland condition represent abstract elliptic operators on the space of leaves, in the sense that they induce a \(K\)-homology class of the (maximal) \(C^*\)-algebra of the holonomy groupoid of the foliation \(H^0\).

All these constructions can be done when adding vector bundles over \(M\). For the de Rham type operators introduced above and the BGG-type operators that we want to define, the vector bundles involved are tractor bundles. They therefore are themselves filtered. Building on \cite{dave2017graded}, we modify a second time the notion of order to take into account these other filtrations. This leads to a transverse graded Rockland condition in Section \ref{Graded order}. We also adapt this condition to treat, not only operators but sequences of operators as well in Section \ref{Rockland Sequences}.

\section{Graded order in the filtered calculus}\label{Graded order}

In this section we combine the results of \cite{cren2022transverse} with the setting of graded filtered calculus of \cite{dave2017graded}. A \(\g\)-module $\E$ is filtered if it admits a (finite) filtration $\cdots \supseteq \E^{j-1}\supseteq \E^{j} \supseteq \cdots$ into subspaces such that $\g_i\E^j \subset \E^{i+j}$. This is always the case if $\E$ is a finite dimensional $\g$-module for a $|k|$-graded Lie algebra $\g$. Indeed there is then an element $X_0 \in \g_0$ such that $\ad(X_0)_{|\g_i} = i\Id_{\g_i}$, diagonalizing the image of $X_0$ gives the decomposition of $\E$. This filtration of \(E\) makes it a filtered \(\g\)-module.
If a vector space is filtered, we denote by $\gr(\E)$ its associated graded vector space, i.e. $\gr(\E) = \bigoplus_p \gr_p(\E)$ with $\gr_p(\E) = \faktor{\E^p}{\E^{p+1}}$. The graded vector space is naturally filtered and $\gr(\gr(\E)) = \gr(\E)$. If $\E$ and $\mathbb{F}$ are filtered and $f\colon \E \to \mathbb{F}$ preserves the filtration then there is a naturally defined linear map $\gr(f) \colon \gr(\E) \mapsto\gr(\mathbb{F})$. A splitting of $\E$ is a filtration preserving isomorphism of vector spaces $\gr(\E) \isomto \E$ whose associated graded homomorphism is the identity. In the case of a $\g$-module as above, the subalgebra $\g_-$ acts on $\gr(\E)$ and all the morphisms above preserve the $\g_-$-module structure. All these notions transpose \textit{mutati mutandis} to vector bundles over a topological space.

During this section $M$ denotes a filtered manifold with 
$$H^{-1} \subset \cdots \subset H^{-r} = TM$$
and $E,F$ vector bundles over $M$, graded in the previous sense. A pseudodifferential operator of graded filtered order $m$ is an operator $P \colon \Gamma(E) \to \Gamma(F)$ such that for any splittings $S_E,S_F$ of $E$ and $F$ we have\footnote{For differential operator, replace $\Psi^m_H$ by the differential operators of order $m$ in the filtered calculus.}:
$$\forall p,q, (S_F^{-1}PS_E)_{q,p} \in \Psi^{m+q-p}_H(M;\gr_p(E),\gr_q(F)).$$
It is actually enough to show this for a single choice of splittings. We denote by $\tilde{\Psi}^m_H(M;E,F)$ the space of such operators. We can then define their principal symbol as:
$$\tilde{\sigma}^m(P) = \sum_{p,q}\sigma^{m+q-p}(S_F^{-1}PS_E) \in \bigoplus_{p,q}\Sigma^{m+q-p}(T_HM;\gr_p(E),\gr_q(F)),$$
the later space will be denoted $\widetilde{\Sigma}^m(T_HM;\gr(E),\gr(F))$. This symbol does not depend on the choice of splitting. Since $\gr(E)$ and $\gr(F)$ are graded they are endowed with a family of inhomogeneous dilations in the same way $T_HM$ is. We denote them by $\delta_{\lambda}^E$ and $\delta_{\lambda}^F$ respectively, $\lambda > 0$. The principal symbols of graded order $m$ are then equivalence classes of kernels $k\mod \Gamma^{\infty}_c(T_HM, \pi^*(\gr(E)^*\otimes \gr(F)))$ (with \(\pi\colon T_HM \to M\)), that satisfy the homogeneity condition:
$$\forall \lambda> 0, \delta_{\lambda *}k \circ \delta_{\lambda}^E = \lambda^m \delta_{\lambda}^F \circ k \mod \Gamma^{\infty}_c(T_HM, \pi^*(\gr(E)^*\otimes \gr(F))).$$

For differential operators, recall that the filtration of $TM$ induced a symbol map $\Diff_H(M) \to \mathcal{U}(\gt_HM)$ which was filtration preserving. In regard of this new notion of order, when filtered vector bundles are involved, one should consider $\mathcal{U}(\gt_HM) \otimes \hom(\gr(E),\gr(F))$ instead, with the graduation given by the graded tensor product, i.e. its $k$-th stratum is:
$$\bigoplus_{p,q} \mathcal{U}(\gt_HM)_{-k+q-p} \otimes \hom(\gr_p(E),\gr_q(F)).$$

It is shown in \cite{dave2017graded} that this calculus satisfies all the usual properties of a pseudodifferential calculus (one basically goes back to the same properties for the usual filtered calculus). In particular they construct an appropriate Sobolev scale and get continuity results for operators in this calculus.

An operator $P \in \widetilde{\Psi}_H^m(M;E,F)$ is said to be graded Rockland if for every $x \in M$ and every non-trivial irreducible unitary representation of the fiber $\pi \in \widehat{T_{H,x}M}\setminus \{1\}$, the operator
$$\diff \pi (\tilde{\sigma}^m(P)) \colon \mathcal{H}_{\pi}^{\infty}\otimes \gr(E_x) \to \mathcal{H}_{\pi}^{\infty}\otimes \gr(F_x)$$
is left injective.

Finally for differential operators, the notion of graded order is simpler. Indeed, let $D \colon \Gamma(M,E) \to \Gamma(M,F)$ be a differential operator of graded filtered order $m$. Since there are no differential operator of negative order then $(S_F^{-1}DS_E)_{q,p} = 0$ whenever $m+q-p<0$. Consequently $D$ maps $\Gamma(M,E^p)$ to $\Gamma(M,F^{p-m})$ and we get an associated graded operator
$$\gr_k(D) \colon \Gamma(M,\gr_{\bullet}(E)) \to \Gamma(M,\gr_{\bullet -m}(F)).$$
Moreover this operator is tensorial and its corresponding vector bundle homomorphism is the direct sum $\bigoplus_p \tilde{\sigma}^m(D)_{p-m,p}$.
In particular a differential operator of graded order $0$ preserves the filtration and we will denote by $\widetilde{D}\colon \gr(E)\to \gr(F)$ the homomorphism constructed from the symbol. In the foliated case, when writing $\gr$ for bundles involving $TM$, we will consider $H^0$ as part of $H^1$ so that the associated graded is a bundle of nilpotent Lie algebras\footnote{See \cite{cren2022transverse}, if we consider $H^0$ in the grading process, the groupoid becomes $\ho(H^0)\ltimes T_{\sfrac{H}{H^0}}M$.}, i.e. $\gr(TM) = \gt_HM$.

Now if $M$ has a foliated filtration, all those notions from \cite{dave2017graded} carry out to transversal symbols and we can define the classes of transverse principal symbols $\widetilde{\Sigma}^m(T_{\sfrac{H}{H^0}}M;\gr(E),\gr(F))$. We also get as in \cite[Section~3.2]{cren2022transverse} restriction maps:
$$\int_{H^0} \colon \widetilde{\Sigma}^m(T_HM;\gr(E),\gr(F)) \to \widetilde{\Sigma}^m(T_{\sfrac{H}{H^0}}M;\gr(E),\gr(F))$$
that are compatible with the product of symbols.

\begin{Definition}An operator $P \in \tilde{\Psi}^m_H(M;E,F)$ is transversally graded Rockland if for every $x \in M$ and $\pi \in \widehat{T_{\sfrac{H}{H^0}}M}\setminus\{1\}$ 
$$\diff\pi\left(\int_{H^0}\tilde{\sigma}_x^m(P)\right) \colon \mathcal{H}_{\pi}^{\infty} \otimes \gr(E_x) \to \mathcal{H}_{\pi}^{\infty} \otimes \gr(F_x),$$
is injective.
\end{Definition}

\section{Graded (transversal) Rockland sequences}\label{Rockland Sequences}

The differential operators we want to analyse are not single operators but a sequence of them. We need to replace the Rockland condition on operators by a condition on sequences of operators. This was done in \cite{dave2017graded}, in the same way elliptic complexes are defined for the usual pseudodifferential calculus (see \cite{AtiyahBott}). Let us consider a sequence of vector bundles $E_i\to M$ and pseudodifferential operators $A_i$ of order $k_i$ in the filtered calculus:
$$\xymatrix{\cdots \ar[r]^-{A_{i-2}} & \Gamma(M,E_{i-1}) \ar[r]^-{A_{i-1}} & \Gamma(M,E_i) \ar[r]^-{A_i} & \Gamma(M,E_{i+1}) \ar[r]^-{A_{i+1}} & \cdots.}$$
This sequence is a Rockland sequence if for every $x \in M$ and every non-trivial irreducible unitary representation $\pi \in \widehat{T_{H,x}M}\setminus\{1\}$ the sequence:

$$\xymatrix{\cdots \ar[r] & \mathcal{H}_{\pi}^{\infty}\otimes E_{i,x} \ar[rr]^-{\diff\pi(\sigma^{k_{i}}_x(A_{i}))} && \mathcal{H}_{\pi}^{\infty}\otimes E_{i+1,x} \ar[r] & \cdots}$$
is weakly exact. This means that the image of an arrow is contained and dense in the kernel of the next one. We can replace the Rockland condition with the graded one, replacing the condition on the $\sigma^{k_i}(A_i)$ by the same condition of weak exactness for the sequence of operators: 
$$\xymatrix{\cdots \ar[r] & \mathcal{H}_{\pi}^{\infty}\otimes \gr(E_{i,x}) \ar[rr]^-{\diff\pi(\tilde{\sigma}^{k_{i}}_x(A_{i}))} && \mathcal{H}_{\pi}^{\infty}\otimes \gr(E_{i+1,x}) \ar[r] & \cdots.}$$

Like before we can adapt this definition with the transversal symbols. We now take $M$ to be a foliated filtered manifold.

\begin{Definition}A sequence of operators
$$\xymatrix{\cdots \ar[r]^-{A_{i-2}} & \Gamma(M,E_{i-1}) \ar[r]^-{A_{i-1}} & \Gamma(M,E_i) \ar[r]^-{A_i} & \Gamma(M,E_{i+1}) \ar[r]^-{A_{i+1}} & \cdots}$$
with $A_i \in \widetilde{\Psi}^{k_i}_H(M,E_{i},E_{i+1})$ is transversally graded Rockland if for every $x \in M$ and $\pi \in \widehat{T_{\sfrac{H}{H^0},x}M}\setminus\{1\}$ the sequence
$$\xymatrix{\cdots \ar[r] & \mathcal{H}_{\pi}^{\infty}\otimes \gr(E_{i,x}) \ar[rr]^-{\diff\pi\int_{H^0}\tilde{\sigma}^{k_{i}}_x(A_{i})} && \mathcal{H}_{\pi}^{\infty}\otimes \gr(E_{i+1,x}) \ar[r] & \cdots}$$
is weakly exact.
\end{Definition}

In particular for a sequence of operators to be (transversal/graded) Rockland, a necessary condition is that the sequence of operators obtained from the (transversal/graded) symbols has to be a complex.

\section{Graded transverse Rockland property of the transverse complex}

Let $G$ be a semi-simple Lie group, $P$ a parabolic subgroup and $(M,F)$ a manifold with transverse parabolic $(G,P)$-geometry. We consider $\E$ a $G$-representation and $E \to M$ the associated bundle. We take $\nabla$ the corresponding tractor connection on $E$ and the sequence of operators 
$$(\Omega^{\bullet,0}(M,E),\dn_N),$$ 
introduced in Section \ref{TransverseComplex}.

The goal of this section is to prove that the transverse complex is graded transversally Rockland. To do this we will identify the graded symbol of $\diff^{\nabla}$ to the differential of the cohomological complex of $\g_-$ acting on $\E$ and will need the regularity assumption. We will prove more general results for foliated filtered manifolds with graded vector bundle and connections in the spirit of \cite{dave2017graded}. Those results will hold under some assumptions that will automatically hold for regular transverse parabolic geometries when the vector bundle is a tractor bundle and the connection the associated tractor connection.

Let $M$ be a foliated filtered manifold, $E \to M$ a filtered vector bundle and $\nabla$ a linear connection on $E$ that preserves the filtration. By that we mean that if $X \in \Gamma(H^p), p\leq 0$ and $\xi\in E^q$ then $\nabla_X\xi \in \Gamma(E^{p+q})$\footnote{Recall that here the convention is that the $H^p$ are defined for non-positive $p$, thus the connection lowers the degree in the fibers of $E$.}. Define $\omega := \gr(\nabla) \in \Gamma(M,\gt^*_{H}M \otimes \gr(E))$\footnote{For transverse parabolic geometries it corresponds to $\rho \circ \omega$ where $\rho \colon \g \to \End(\E)$ is the representation inducing the tractor bundle.}. Since $\nabla$ is filtration preserving, $\omega$ takes values in the space of elements of degree $0$ (for the graduation of the graded tensor product). Sections of $H^0$ are considered as degree 1 so their image under $\omega$ vanishes. We can thus push it forward to $\overline{\omega} \in \Gamma(M,\gt^*_{\sfrac{H}{H^0}}M \otimes \gr(E))$. More generally if $A \colon E \to T^*M \otimes E$ preserves the filtration then images of sections of $H^0$ vanish under $\gr(A)$ so we can factor it to 
$$\overline{\gr(A)} \colon \gr(E) \to \gt_{\sfrac{H}{H^0}}^*M\otimes \gr(E).$$
We can extend the construction of the transverse complex of Section \ref{TransverseComplex} in this case so we fix a subbundle $N \subset TM$ complementary to $H^0$. Since $\nabla$ is filtration preserving, the operators $\diff_N^{\nabla}$ are differential operators of graded filtered order $0$.

We have identifications 
$$\gr(\Lambda^kT^*M\otimes E) \cong \Lambda^k\gt_H^*M \otimes \gr(E)$$
and 
$$\gr(\Lambda^{k,0}T^*M \otimes E) \cong \Lambda^k \gt^*_{\sfrac{H}{H^0}}M \otimes \gr(E).$$
Let $x \in M$,we have the identification:
$$\CCC^{\infty}(T_{\sfrac{H}{H^0},x}M,\gr(\Lambda^{k,0}T^*_xM\otimes E_x)) \cong \Omega^k(T_{\sfrac{H}{H^0},x}M,\gr(E_x))$$
because the tangent bundle of the group $T_{\sfrac{H}{H^0},x}M$ is trivial.

Finally to compute the symbol, we will use the following fact. If \(X \in \Gamma(H^{-k})\subset \mathfrak{X}(M)\), we can consider \(\nabla_X\) as a filtered operator of graded order \(k\). Then we have \(\tilde{\sigma}^k(\nabla_X) = (X\mod H^{-k+1}) \otimes \Id_E\) and the same formula holds for \(\int_{H^0}\tilde{\sigma}^k(\nabla_X)\) (replacing \(X\) by \(X \mod H^0\) if \(k=1\)).

\begin{Lemma}Under the identification 
$$\CCC^{\infty}(T_{\sfrac{H}{H^0},x}M,\gr(\Lambda^{k,0}T^*_xM\otimes E_x)) \cong \Omega^k(T_{\sfrac{H}{H^0},x}M,\gr(E_x))$$
we have:
$$\int_{H^0}\tilde{\sigma}^{0}_x(\diff^{\nabla}_N) = \diff + \overline{\omega}\wedge_x\cdot.$$
\end{Lemma}
\begin{proof}
Let \(k\geq 0\), \(\eta \in \Omega^{k,0}(M,E)\), and \(X_1,\cdots,X_{k+1} \in \mathfrak{X}(M)\). We have the formula:
\begin{align*}
\dn &\eta (X_1,\cdots,X_{k+1}) = \sum_{j = 1}^{k} (-1)^{j+1} \nabla_{X_j}\eta(X_1,\cdots,\widehat{X_j},\cdots,X_{k+1}) \\
							  &+ \sum_{1\leq i < j\leq k+1}(-1)^{i+j} \eta([X_i,X_j],X_1,\cdots,\widehat{X_i},\cdots,\widehat{X_j},\cdots,X_{k+1}).
\end{align*}
We get a similar formula for \(\dn_N\) by replacing all the vectors above by their transverse part (after choosing a complementary bundle \(N\subset TM, N \oplus H^0 = TM\)). We now freeze the coefficients at \(x\in M\) and use the identification mentioned above: \(\CCC^{\infty}(T_{\sfrac{H}{H^0},x}M,\gr(\Lambda^{k,0}T^*_xM\otimes E_x)) \cong \Omega^k(T_{\sfrac{H}{H^0},x}M,\gr(E_x))\). Let \(\mu \in \Lambda^k\ttt^*_{\sfrac{H}{H^0},x}M\otimes \gr(E_x)\) and \(\xi_1, \cdots, \xi_{k+1} \in \ttt_{H,x}M\) considered as invariant form and vector fields on \(T_{\sfrac{H}{H^0}}M\). We obtain:
\begin{align*}
\tilde{\sigma}^0_x(\dn_N) &\mu (\xi_1,\cdots,\xi_{k+1}) = \sum_{j = 1}^{k} (-1)^{j+1} \bar{\omega}_x(\xi_j)\mu(\xi_1,\cdots,\widehat{\xi_j},\cdots,\xi_{k+1}) \\
							  &+ \sum_{1\leq i < j\leq k+1}(-1)^{i+j} \mu([\xi_i,\xi_j],\xi_1,\cdots,\widehat{\xi_i},\cdots,\widehat{\xi_j},\cdots,\xi_{k+1}).
\end{align*}
Therefore \(\tilde{\sigma}^0_x(\dn) (\mu) = (\diff + \bar{\omega}_x\wedge)(\mu)\) and we get the result we wanted.
\end{proof}

Since $\dn_N$ is a differential operator of order $0$ we can also consider $\gr(\dn_N)$ and taking its quotient $\partial^{\overline{\omega}} := \overline{\gr(\dn)} = \gr(\dn_N)$ as a morphism 
$$\partial^{\overline{\omega}}\colon \Lambda^k\gt_{\sfrac{H}{H^0}}^*M \otimes \gr(E) \to \Lambda^{k+1}\gt_{\sfrac{H}{H^0}}^*M \otimes \gr(E).$$

\begin{Proposition}For every $x \in M$, 
$$(\Lambda^{\bullet}\gt_{\sfrac{H}{H^0},x}^*M \otimes \gr(E_x),\partial^{\overline{\omega_x}}),$$
is a Chevalley-Eilenberg type sequence of operators associated to the map $\overline{\omega}_x$.\end{Proposition}
\begin{proof}
By construction $\partial^{\overline{\omega}}$ is the unique extension to higher exterior powers of the action of $\omega_x$ hence it is the Chevalley-Eilenberg differential obtained from this map.
\end{proof}

We did not call this sequence of operators a complex because it is a complex if and only if $\overline{\omega_x}$ is a representation of $\gt_{\sfrac{H}{H^0},x}M$ on $\gr(E_x)$. This leads us to the necessary condition for the sequence to be transversally graded Rockland: the sequence of symbols has to be exact.

\begin{Proposition}For each $x\in M$ we have the equivalence:
\begin{itemize}
\item[a)] The (2,0)-part of the curvature\footnote{We only showed that the curvature had this bi-degree for transverse parabolic geometries. This does not seem to hold in general.}, $(R^{\nabla})^{(2,0)} \in \Omega^{2,0}(M,\End(E))$ has degree $1$, i.e. if  $x\in M$, $X \in H_x^i, Y \in H_x^j, \psi \in E_x^p$ then 
$$R^{\nabla}(X,Y)\psi \in E_x^{p+i+j+1}.$$
\item[b)] $\int_{H^0}\tilde{\sigma}_x^{0}(\dn_N)^2 = 0$.
\item[c)] $(\partial^{\overline{\omega_x}})^2 = 0$.
\item[d)] $\overline{\omega_x}$ is a representation of $\gt_{\sfrac{H}{H^0},x}M$ on $\gr(E_x)$ that preserves the graduation.
\end{itemize}
\end{Proposition}
\begin{proof}
To make the link with the curvature recall that 
$$(\dn_N)^2 = (R^{\nabla})^{(2,0)}\wedge\cdot - (\diff^{\nabla}_F \iota_{\Theta} + \iota_{\Theta}\diff^{\nabla}_F).$$
If we show that the graded principal symbol of $(\diff^{\nabla}_F \iota_{\Theta} + \iota_{\Theta}\diff^{\nabla}_F)$ vanishes then we get the equivalence between $a),b)$ and $c)$ by the previous lemma and proposition. The equivalence $c) \Leftrightarrow d)$ is immediate:
$$(\partial^{\overline{\omega}_x})^2(X,Y) = [\omega(X),\omega(Y)] - \omega([X,Y]).$$
Now since $\nabla$ preserves the filtration, elements of degree $0$ do not change the degree. However $\Theta$ takes value in sections of the foliation, thus in elements of degree $0$. Therefore, the terms other than $(R^{\nabla})^{(2,0)}$ vanish when we take the associated graded morphism and the transverse symbol.
\end{proof}

Under these conditions we can state the main result:

\begin{Theorem}\label{TrsvRock}
Let $(M,H)$ be a foliated filtered manifold $E \to M$ a filtered vector bundle and $\nabla$ a $\ho(H^0)$-equivariant connection on $E$ that preserves the filtration and such that the $(2,0)$-component of its curvature has degree $1$. Then the sequence of operators given by the transverse de Rham sequence $(\Omega^{\bullet,0}(M,E),\dn_N)$ is transversally graded Rockland.
\end{Theorem}

\begin{proof}
We have already computed the graded transverse symbol of $\dn_N$. Under those assumptions we get $\int_{H^0}\tilde{\sigma}^{0}(\dn_N)^2 = 0$. Let $x\in M$ and take a non-trivial irreducible unitary representation of the fiber group $\pi \in \widehat{T_{\sfrac{H}{H^0},x}M}\setminus\{1\}$. The symbolic complex 
\[\left(\Lambda^{\bullet}\gt^*_{\sfrac{H}{H^0},x}M\otimes E_x \otimes \mathcal{H}_{\pi}^{\infty},\diff\pi\left(\int_{H^0}\tilde{\sigma}^{0}(\dn_N)\right)\right),\]
is the Chevalley-Eilenberg complex of the $\gt_{\sfrac{H}{H^0},x}M$-module $ E_x \otimes \mathcal{H}_{\pi}^{\infty}$. We need to show that it is an exact sequence, i.e. that its cohomology is trivial.
\end{proof}

\begin{Lemma}Let $N$ be a finite dimensional, connected, simply connected, nilpotent Lie group and $\n$ its Lie algebra. Let $\pi \in \hat{N}\setminus \{1\}$ be a non-trivial irreducible unitary representation and $V$ be a finite dimensional representation of $N$. Then the associated Chevalley-Eilenberg cohomology vanishes: $H^*(\n,\mathcal{H}_{\pi}^{\infty}\otimes V) = 0$.
\end{Lemma}

\begin{proof}
We first reduce to the case $V = 1$. Since $\n$ is nilpotent, using Engel's theorem we can find an invariant subspace $W \subset V$ of dimension 1. The action on $W$ is then trivial and the long exact sequence in cohomology becomes
\begin{equation*}
\resizebox{.9\hsize}{!}{$\xymatrix{\cdots \ar[r] & H^k(\n,\mathcal{H}_{\pi}^{\infty}\otimes V) \ar[r] & H^k(\n,\mathcal{H}_{\pi}^{\infty}\otimes \faktor{V}{W}) \ar[r] & H^{k+1}(\n,\mathcal{H}_{\pi}^{\infty}) \ar[r] & \cdots}$}
\end{equation*}
It thus suffices to prove the result for $V = 1$ and we get every finite dimensional representations by induction. Now for the trivial representation we use induction on the dimension of $\n$. Since $\n$ is nilpotent it has a non-trivial center \cite{Kirillov}, we take $\mathfrak{z} \subset \n$ a $1$-dimensional central subalgebra. Now the Hochschild-Serre exact sequence \cite{HochschildSerre} has its $E_2$-page equal to 
$$E_2^{p,q} = H^p(\faktor{\n}{\mathfrak{z}},H^q(\mathfrak{z},\mathcal{H}_{\pi}^{\infty}))$$
and converges to $H^*(\n,\mathcal{H}_{\pi}^{\infty})$. Now because $\pi$ is an irreducible representation of $G$ the action of $\mathfrak{z}$ on $\mathcal{H}_{\pi}^{\infty}$ is by scalar operators (the corresponding scalar is called the infinitesimal character, see \cite{Kirillov}). If the action of $\mathfrak{z}$ is non-trivial then the cohomology vanishes. Indeed, $H^0(\mathfrak{z},\mathcal{H}_{\pi}^{\infty})$ is the space of vectors that vanish under the action of $\mathfrak{z}$ hence $\{0\}$. The space $H^1(\mathfrak{z},\mathcal{H}_{\pi}^{\infty})$ is the space of $\mathcal{H}_{\pi}^{\infty}$-valued outer derivations and is thus trivial (if $D$ is such a derivation, we can write $Dx = \frac{1}{\lambda}\diff\pi(Z)(Dx)$ where $Z \in \mathfrak{z}\setminus\{0\}$ satisfies $\diff \pi(Z) = \lambda\Id_{\mathcal{H}_{\pi}^{\infty}}$ thus $D$ is inner).
In the trivial case we have 
$$H^*(\mathfrak{z},\mathcal{H}_{\pi}^{\infty}) = \mathcal{H}_{\pi}^{\infty} \oplus \mathcal{H}_{\pi}^{\infty}$$ 
as a $\faktor{\n}{\mathfrak{z}}$-module. To proceed with the induction, denote by $Z \subset N$ the closed subgroup of $N$ with Lie algebra $N$. The group $Z$ acts trivially on $\mathcal{H}_{\pi}^{\infty}$ and thus on  $\mathcal{H}_{\pi}$ by continuity. The representation $\pi$ factors to a non-trivial irreducible representation of $\faktor{N}{Z}$ whose space of smooth vector fields is $\mathcal{H}_{\pi}^{\infty}$. By induction on the dimension of $\n$ we have $H^*(\faktor{\n}{\mathfrak{z}},\mathcal{H}_{\pi}^{\infty}) = 0$ and thus the same result for $\n$ using the spectral sequence. Indeed the $E_2$-page becomes either directly $0$ or $H^p(\faktor{\n}{\mathfrak{z}},\mathcal{H}_{\pi}^{\infty})$ for $q = 0,1$ and thus still $0$, the $E_2$ page vanishes and so does the cohomology $H^*(\n,\mathcal{H}_{\pi}^{\infty})$.
\end{proof}

\section{Transverse BGG operators for transverse parabolic geometries}

As a particular case of Theorem \ref{TrsvRock} we get:

\begin{Proposition}Let $M$ be a foliated manifold with transverse parabolic $(G,P)$-geometry. Let $\E$ be a $G$-representation and $E \to M$ the associated tractor bundle with the associated tractor connection $\nabla$. If the geometry is regular then the transverse complex $(\Omega^{\bullet,0}(M,E),\dn_N)$ is transversally graded Rockland.
\end{Proposition}
\begin{proof}
The condition on the tractor connection and its curvature are consequences of Propositions \ref{Equivariance} and \ref{CurvatureRegular} respectively.
\end{proof}

We now want to study in more depth this sequence of operators and define an analog of the BGG sequences using the appropriate version of a transverse Laplacian. This could be done in more generality with more hypotheses that would automatically be satisfied for regular transverse parabolic geometries. We refer to \cite{dave2017graded} where this general machinery is described in the non-foliated case. Other examples could include a transverse version of the Rumin complex and its corresponding hypoelliptic Laplacian \cite{Rumin,RuminSeshadri}.

To do this recall there is the co-differential on $\Lambda^{\bullet}\left(\faktor{\g}{\p}\right)^*$ inherited from the isomorphism of $\p$-modules between $\left(\faktor{\g}{\p}\right)^*$ and $\p_+$. Using the isomorphism $\faktor{TM}{H^0} \cong M \times \faktor{\g}{\p}$ given by the Cartan connection we can define a bundle map called the Kostant co-differential:
$$\partial^* \colon \Lambda^{\bullet,0}T^*M \otimes E \to \Lambda^{\bullet-1,0}T^*M \otimes E.$$
The algebraic co-differential was $P$-equivariant so $\partial^*$ is $P$-equivariant. Since the identification between $\faktor{TM}{H^0}$ and $M\times\faktor{\g}{\p}$ is done using the Cartan connection $\omega$ the map $\partial^*$ is also $\ho(\ker(\omega))$-equivariant and thus $\partial^*$ is $\ho(H^0)$-equivariant. 

We now consider the differential operators of graded order $0$:
$$\Box_{\bullet} := \diff^{\nabla}_N\partial^* + \partial^*\diff^{\nabla}_N \colon \Omega^{\bullet,0}(M,E) \to  \Omega^{\bullet,0}(M,E).$$
Since $\overline{\gr(\dn_N)} = \partial_{\g_-}$ we get that $\forall x\in M, \overline{\gr(\Box_{\bullet})}_x = \Box_{\g_-,\bullet}$ where  $\Box_{\g_-,\bullet}$ is Kostant's Laplacian. Remember from Section \ref{Algebraic} that we have a Hodge decomposition and we can transpose it to the bundle level. Denote by 
$$\widetilde{P}_k \in \End(\Lambda^k \gt^*_{\sfrac{H}{H^0}}M\otimes \gr(E)),$$ 
the projection onto the generalised zero eigenspace of $\overline{\gr(\Box_k)}$. On each fiber $\tilde{P}_k$ corresponds to the same type of projection constructed at the algebraic level. In particular $\tilde{P}_k$ is $\ho(H^0)$-equivariant.

\begin{Remark}As described earlier, the graduation on $\E$ comes from the action of a particular element of $\g_0$ (which was unique) and $\E$ itself is written as a graded sum of subspaces. Consequently this grading is $G_0$-equivariant, where $G_0 \subset G$ integrates $\g_0$. Moreover we have $G_0 = \faktor{P}{P_+}$ thus $P$ preserves the filtration on $\E$ and acts trivially on the associated graded $\gr(\E)$. Consequently we have $\gr(E) \cong \E$ as $P$-representations. Therefore at the bundle level we have $\gr(E) = E$ as $\ho(H^0)$-equivariant bundles. In the case of regular transverse parabolic geometries we will therefore write $E$ instead of $\gr(E)$ for a tractor bundle.
\end{Remark}

The homology spaces $H_*(\p_+,\E)$ are $\g_0$-modules. We extend their structure to $\p$-modules with a trivial action of $\p_+$. We can thus define 
$$\mathcal{H}_* := M \times H_*(\p_+,\E) \to M,$$ 
as a $\ho(H^0)$-equivariant bundle. Using Kostant's co-differential we get natural $\ho(H^0)$-equivariant bundle maps $\pi \colon \ker(\partial^*_k)\to \mathcal{H}_k$. We can use the same arguments as in \cite{dave2017graded} to construct differential operators 
$$P_{\bullet} \colon \Omega^{\bullet,0}(M,E) \to \Omega^{\bullet,0}(M,E).$$ 
They satisfy the following properties: $P_{\bullet}^2 = P_{\bullet}, P_{\bullet}\Box_{\bullet} = \Box_{\bullet}P_{\bullet}$, $P_{\bullet}$ preserves the filtration of the corresponding bundle and $\gr(P_{\bullet}) = \tilde{P_{\bullet}}$. By construction of $\tilde{P_{\bullet}}$ we have that $\Box_{\bullet}$ is nilpotent on $\im(P_{\bullet})$ and invertible on $\ker(P_{\bullet})$. Choose a splitting of $\gt_{\sfrac{H}{H^0}}M$. This induces splittings $S_{\bullet}$ of all the bundles $\Lambda^{\bullet}\gt_{\sfrac{H}{H^0}}^*M\otimes E$. We can now relate the decompositions induced by $P_{\bullet}$ and $\tilde{P_{\bullet}}$ with the operators:
$$L_{\bullet} = P_{\bullet}S_{\bullet}\tilde{P_{\bullet}} + (1-P_{\bullet})S_{\bullet}(1-\tilde{P_{\bullet}}) \colon \Gamma(\Lambda^{\bullet}\gt^*_{\sfrac{H}{H^0}}M\otimes E) \isomto \Omega^{\bullet,0}(M,E).$$
These operators preserve the decompositions induced by $P_{\bullet}$ and $\tilde{P_{\bullet}}$ at the level of sections. They thus induce differential operators:
$$\bar{L_{\bullet}} \colon \Gamma(\mathcal{H}_{\bullet}) \to \im(P_{\bullet}),$$
with $\bar{L_{\bullet}}\pi_{\bullet} = P_{\bullet}$ and $\pi_{\bullet} \bar{L_{\bullet}} = 1$ where $\pi_{\bullet}$ are the natural quotient maps onto $\mathcal{H}_{\bullet}$.

\begin{Definition}The transverse BGG operators are the differential operators of graded filtered order $0$: 
$$D_{\bullet} := \pi_{\bullet+1} P_{\bullet+1} \dn_N L_{\bullet} \colon \mathcal{H}_{\bullet} \to \mathcal{H}_{\bullet}.$$
\end{Definition}

In the case of a non-foliated Cartan geometry (i.e. the leaves are the $P$-orbits of a free $P$-action), the space $\faktor{M}{P}$ is endowed with a parabolic geometry in the usual sense. In this case, since all the vector bundles used are $P$-equivariant, they descend to (locally trivial) vector bundles on $\faktor{M}{P}$. The same goes for the operators $D_{\bullet}$ and the resulting sequence of operators on $\faktor{M}{P}$ is the curved BGG sequence considered in \cite{CurvedBGGArticle,dave2017graded}.

\begin{Theorem}The BGG sequence $(\mathcal{H}_{\bullet},D_{\bullet})$ is transversally graded Rockland. Moreover if $(\Omega^{\bullet,0}(M,E),\dn_N)$ is a co-chain complex\footnote{This is the case for instance if there is an involutive complementary subbundle $N \subset TM$ to the foliation and the curvature $K$ vanishes.} then so is $(\mathcal{H}_{\bullet},D_{\bullet})$. The operators $\bar{L_{\bullet}}$ then define a chain map and induces an isomorphism between the cohomology of both complexes.
\end{Theorem}
\begin{proof}
Since $\int_{H^0}\tilde{\sigma}^{0}(\dn_N)^2 = 0$ then we have: 
$$\int_{H^0}\left(\tilde{\sigma}^{0}(\dn_N)\tilde{\sigma}^{0}(\Box_{\bullet})\right) = \int_{H^0}\left(\tilde{\sigma}^{0}(\Box_{\bullet+1})\tilde{\sigma}^{0}(\dn_N)\right).$$ 
Therefore by construction of the projectors we also get the commutation relations $\int_{H^0}\left(\tilde{\sigma}^{0}(\dn_N)\tilde{\sigma}^{0}(P_{\bullet})\right) = \int_{H^0}\left(\tilde{\sigma}^{0}(P_{\bullet+1})\tilde{\sigma}^{0}(\dn_N)\right)$. In particular at the symbolic level the symbols $\int_{H^0}\tilde{\sigma}^{0}(L_{\bullet+1}^{-1}\dn_NL_{\bullet})$ are diagonal with respect to the decomposition $\Gamma(\im(\tilde{P_{\bullet}})) \oplus \Gamma(\ker(\tilde{P_{\bullet}}))$. Under the isomorphism $\Gamma(\mathcal{H}_{\bullet})\cong \Gamma(\im(\tilde{P_{\bullet}}))$, $\int_{H^0}\tilde{\sigma}^{0}(D_{\bullet})$ appears as the one of the two diagonal terms of $\int_{H^0}\tilde{\sigma}^{0}(L_{\bullet+1}^{-1}\dn_NL_{\bullet})$. Therefore the BGG sequence $(\mathcal{H}_{\bullet},D_{\bullet})$ is also transversally graded Rockland.
If $(\Omega^{\bullet,0}(M,E),\dn_N)$ is a complex, then by construction so is $(\mathcal{H}_{\bullet},D_{\bullet})$. Moreover in this case the operator $\dn_N$ itself is diagonal (using the same arguments as before). We have 
$$L_{\bullet+1}^{-1}\dn_N L_{\bullet} = \tilde{D}_{\bullet} \oplus A_{\bullet},$$
and both $(\Gamma(\im(\tilde{P_{\bullet}})),\tilde{D}_{\bullet})$ and $(\Gamma(\ker(\tilde{P_{\bullet}})),A_{\bullet})$ are complexes. Moreover, the complex $(\Gamma(\im(\tilde{P_{\bullet}})),\tilde{D}_{\bullet})$ is isomorphic to the BGG complex so we need to show that the chain map 
$$L_{\bullet} \colon (\Gamma(\im(\tilde{P_{\bullet}})),\tilde{D}_{\bullet})\to (\Omega^{\bullet,0}(M,E),\dn_N),$$ 
is invertible up to homotopy. Define
\begingroup\makeatletter\def\f@size{11}\check@mathfonts $$G_{\bullet} = A_{\bullet-1}(1-\tilde{P_{\bullet-1}}) \gr(\partial^*_{\bullet})\gr(\Box_{\bullet})^{-1} + (1-\tilde P_{\bullet})\gr(\partial^*_{\bullet+1})\gr(\Box_{\bullet+1})^{-1}\gr(\dn_N)$$ \endgroup
as a differential operator of graded filtered order $0$ on $\ker(\tilde{P_{\bullet}})$. We have $\gr(G_{\bullet}) = 1$, thus \(G_{\bullet}\) is invertible. This operator conjugates the complex $(\Gamma(\ker(\tilde{P_{\bullet}})),A_{\bullet})$ to the acyclic complex $(\Gamma(\ker(\tilde{P_{\bullet}})),\gr(\dn_N))$. This gives the isomorphism in cohomology. We can also write an explicit homotopy. The operator
$$h_{\bullet} := L_{\bullet-1}G_{\bullet-1}(1-\tilde{P_{\bullet-1}})\gr(\partial^*_{\bullet})\gr(\Box_{\bullet})^{-1}G_{\bullet}^{-1}(1-\tilde{P_{\bullet}})L_{\bullet}^{-1},$$
maps $\Omega^{\bullet,0}(M,E)$ to $\Omega^{\bullet-1,0}(M,E)$ and satisfies
$$1-L_{\bullet}\tilde{P_{\bullet}}L_{\bullet}^{-1} = \dn_N h_{\bullet} + h_{\bullet+1} \dn_{N}.$$
Since $\tilde{P_{\bullet}}L_{\bullet}^{-1} L_{\bullet |\Gamma(\im(\tilde{P_{\bullet}}))} = 1$ and $\tilde{D_{\bullet}}\tilde{P_{\bullet}}L_{\bullet}^{-1} = \tilde{P_{\bullet}}L_{\bullet}^{-1} \dn_N$ we get that $\tilde{P_{\bullet}}L_{\bullet}^{-1}$ is the inverse of $L_{\bullet}$ up to homotopy.
\end{proof}


\end{document}